\documentclass[12pt,reqno]{amsart}

\usepackage[latin1]{inputenc}

\usepackage[all]{xy}
\def\theoremnumberstyle{plain}
\usepackage{tom}

\DeclareMathOperator{\lc}{lc}
\DeclareMathOperator{\val}{val}
\DeclareMathOperator{\tin}{t-in}
\DeclareMathOperator{\Trop}{Trop}
\DeclareMathOperator{\tsupp}{t-supp}
\newcommand{\M}{\mathbbm{M}}

\newcommand{\lang}[1]{}
\newcommand{\kurz}[1]{#1}
\newcommand{\bmath}{\kurz{\begin{math}}\lang{\begin{displaymath}} }
\newcommand{\emath}{\kurz{\end{math}}\lang{\end{displaymath}} }

\begin{document}

   \parindent0cm

   \title[Generalised Puiseux Series]{A Field of Generalised Puiseux Series for Tropical Geometry}
   \author{Thomas Markwig}
   \address{Universit\"at Kaiserslautern\\
     Fachbereich Mathematik\\
     Erwin--Schr\"odinger--Stra\ss e\\
     D --- 67663 Kaiserslautern\\
     Tel. +496312052732\\
     Fax +496312054795
     }
   \email{keilen@mathematik.uni-kl.de}
   \urladdr{http://www.mathematik.uni-kl.de/\textasciitilde keilen}
   \thanks{The author was supported by the IMA, Minneapolis.}


   \date{September, 2007.}

     
   \begin{abstract}
     In this paper we define a field $\K$  of characteristic zero with valuation
     whose value group is $(\R,+)$, and we show that this field of
     \emph{generalised Puiseux series} is
     algebraically closed and complete with respect to the norm
     induced by its valuation. We consider this field to be a good candidate for 
     the base field for tropical geometry.
   \end{abstract}

   \maketitle

   In order to study the geometric properties of a variety, say
   $V=V(I)\subseteq(\C^*)^n$ with
   $I\lhd\C[x_1^{\pm 1},\ldots,x_n^{\pm 1}]$, it is common to study as well
   deformations of the variety respectively of its defining equations,
   i.e.\ we replace the ideal $I$ 
   by an ideal $I_t\lhd\C[[t]][x_1^{\pm 1},\ldots,x_n^{\pm 1}]$ such that
   $I=\{f_{|t=0}\;|\;f\in I_t\}$. The generic fibre of this family is
   then defined by the ideal which $I_t$ generates over the quotient
   field $\Quot\big(\C[[t]])$ of the power series ring
   $\C[[t]]$. Unfortunately, this field is not algebraically
   closed. If we are interested in the geometric properties of
   the general fibre it thus is natural to pass to the algebraic closure
   of this field, which is the field
   \begin{displaymath}
     \C\{\{t\}\}=\bigcup_{N=1}^\infty\C\Big[\Big[t^{\frac{1}{N}}\Big]\Big]
     =\left\{\sum_{k=m}^\infty a_k\cdot t^\frac{k}{N}\;\Big|\; m\in\Z,
       N\in\N, a_k\in\C\right\}
   \end{displaymath}
   of Puiseux series over $\C$. This field comes with a valuation
   \begin{displaymath}
     \val:\C\{\{t\}\}^*\longrightarrow\Q:\sum_{k=m}^\infty a_k\cdot
     t^\frac{k}{N}\mapsto \min\left\{\frac{k}{N}\;\Big|\;a_k\not=0\right\}
   \end{displaymath}
   sending a Puiseux series to its order. Given an ideal
   $J\lhd\C\{\{t\}\}[x_1^{\pm 1},\ldots,x_n^{\pm 1}]$ and its variety
   $V(J)\subseteq (\C\{\{t\}\}^*)^n$ the idea of tropical geometry is to
   try to understand  $V(J)$ better by just looking at its image
   under  the $n$-fold Cartesian product of the valuation map
   \begin{displaymath}
     \val:(\C\{\{t\}\}^*)^n\longrightarrow
     \Q^n:(p_1,\ldots,p_n)\mapsto\big(\val(p_1),\ldots,\val(p_n)\big),
   \end{displaymath}
   or rather its closure, say
   $\Trop\big(V(J)\big)$, in $\R^n$ under the Euclidean topology. 
   (Depending on whether they prefer $\max$ over $\min$ people sometimes
   use the negative of this function $\val$ for the process of
   tropicalisation.)  Due to  the Theorem of Bieri-Groves (see
   \cite[Thm.~A]{BG84}, \cite[Thm.~2.4]{SS04a}, \cite[Thm.~2.2.5]{EKL06}) and
   the Lifting Lemma (see \cite[Thm.~2.13]{EKL06},  
   \cite[Thm.~2.1]{SS04a}, \cite[Thm.~4.2]{Dra06},
   \cite[Thm.~2.13]{JMM07})  this object turns out to be piece 
   wise linear and its points, say $\omega$, can be characterised by the
   fact that the $t$-initial ideal of $J$ with respect to $\omega$ is
   monomial free (see e.g.\ \cite{JMM07}). Taking into account how
   crude the valuation map is, that is, how much information it
   ignores  (e.g.\ \cite[Thm.~4.2]{Pay07} shows that each fibre of the
   restriction of $\val$ to $V(J)$ is dense in $V(J)$ as soon as it is
   non-empty), it is  surprising how much valuable information is
   preserved (see e.g.\ \cite{EKL06}, \cite{Spe05}, \cite{Tab05a}, \cite{Dra06},
   \cite{Gat06}, \cite{Mik06}, \cite{Shu06}, \cite{KMM07}, \cite{Boe07}). 

   Forgetting about the motivation why the field $\C\{\{t\}\}$ should
   be an interesting field to start with, one can replace
   $\C\{\{t\}\}$ by any field $K$ with a valuation whose value group
   is dense in $\R$ with respect to the Euclidean topology, and study
   the tropicalisation of varieties $(K^*)^n$ via the $n$-fold Cartesian
   product of the valuation map. The Lifting Lemma holds in any
   case (see \cite[Thm.~2.1]{SS04a}, \cite[Thm.~4.2]{Dra06}), and it
   seems somehow more natural to choose a 
   field where the valuation map is surjective onto $\R$, so
   that $\Trop\big(V(J)\big)$ coincides with
   $\val\big(V(J)\big)$ and no topological closure is
   necessary, which also leads to a larger class of tropical
   varieties, e.g.\ points with non-rational coordinates. To
   this extend other authors (see e.g. \ \cite{Pay07},
   \cite[Chap.~4.2]{Boe07}) use the  
   following field of a generalised Laurent series,
   \begin{displaymath}
     K=\left\{\sum_{\alpha\in A} a_\alpha\cdot t^\alpha\;\Big|\;
     A\subset\R \mbox{ well-ordered}, a_\alpha\in\C\right\},
   \end{displaymath}
   with the obvious addition and multiplication, and where the
   valuation of generalised Laurent series is again given by its
   order. This field is indeed algebraically closed and complete (see
   \cite[Thm.~2]{Ray74}),  and
   its value group is $\R$. However, it seems a rather big step
   from the field $\C\{\{t\}\}$ to this field $K$ by passing to 
   exponent sets $A$ which are arbitrary well-ordered sets. In
   this paper we want to introduce an alternative field $\K$ which contains
   $\C\{\{t\}\}$ and is contained in $K$, which has a valuation with
   value group $\R$ and which is also algebraically closed and complete. In
   comparison with $\C\{\{t\}\}$ it thus has the advantage of
   completeness and that no
   topological closure is necessary when tropicalising, and in
   comparison with $K$ it has the advantage that the exponents of the
   generalised Laurent series considered are simply sequences of real
   numbers diverging to infinity. 

   \begin{definition}\label{def:K}
     \begin{enumerate}
     \item We use the symbol
       \begin{displaymath}
         \alpha_n\nearrow\infty
       \end{displaymath}
       to denote a sequence $(\alpha_n)_{n\in\N}$ of real numbers which
       is strictly monotonously increasing and unbounded, and we call
       the sequence \emph{smiub}. Note that
       such a sequence is determined uniquely by the set
       $\{\alpha_n\;|\;n\in \N\}$. 
     \item  We define the set $\M$ to be
       \begin{displaymath}
         \M=\big\{\{\alpha_n\;|\;n\in\N\}\;\big|\;\alpha_n\nearrow\infty\big\}
         \cup\{A\;|\;A\subset\R, \# A<\infty\},
       \end{displaymath}
       which is basically the union of all \emph{smiub}-sequences and
       of all finite sequences.
     \item Given a set $A\in\M$ and $a_\alpha\in\C^*$ for $\alpha\in
       A$ we use the short hand notation
       \begin{equation}\label{eq:K:1}
         \sum_{\alpha\in A} a_{\alpha}\cdot t^{\alpha}
       \end{equation}
       in order to denote the function
       \begin{displaymath}
         f:\R\longrightarrow\C:\alpha\mapsto
         \left\{
         \begin{array}[m]{ll}
           a_{\alpha}, & \mbox{ if } \alpha\in A,\\
           0,&\mbox{ else,}
         \end{array}
         \right.
       \end{displaymath}
       and we call $A$ the \emph{support} of $f$.
       The set of all function $f:\R\rightarrow\C$ of this type is
       denoted by $\K$, i.e.
       \begin{displaymath}
         \K =\left\{\sum_{\alpha\in A}a_{\alpha}\cdot
             t^{\alpha}\;\Big|\;A\in\M\right\}.
       \end{displaymath}
       Note that we allow the set $A$ to be empty, so that the constant
       zero function is contained in $\K$. We call the elements of
       $\K$ \emph{generalised Puiseux series}.
     \item If $A,B\subset\R$ we set $A*B=\{\alpha+\beta\;|\;\alpha\in A,\beta\in B\}$. 
     \end{enumerate}
   \end{definition}

   \begin{remark}
     \begin{enumerate}
     \item Note that the representation \eqref{eq:K:1} of a function
       $f\in\K$ is unique, and it is either a generalised Laurent polynomial
       \begin{displaymath}
         f=\sum_{n=0}^k a_{\alpha_n}\cdot t^{\alpha_n},
       \end{displaymath}
       or it is a generalised Laurent series
       \begin{displaymath}
         f=\sum_{n=0}^\infty a_{\alpha_n}\cdot t^{\alpha_n},
       \end{displaymath}
       where the exponents are real numbers forming a
       \emph{smiub}-sequence $\alpha_n\nearrow\infty$. In particular
       with the notation from above we obviously have
       \begin{displaymath}
         \C\{\{t\}\}\;\;\subset\K\subset\;\; K.
       \end{displaymath}
       We will, however, not bother too much about the \emph{uniqueness} of
       the representation and spoil it by allowing the coefficients to
       be zero in order make the notation simpler.
     \item For $A,B\in\M$ one easily sees that $A\cup B\in\M$ and
       $A*B\in\M$. Moreover, for any fixed element $\gamma\in A*B$
       there is only a finite number of pairs $(\alpha,\beta)\in A\times B$
       such that $\alpha+\beta=\gamma$.
     \item Sometimes we will have to access the value of a function
       $f\in\K$ for $\alpha=0$ where $f$ is given as an algebraic expression
       involving several elements of $\K$. We then will use the short
       hand notation
       \begin{displaymath}
         f_{|t=0}=f(0).
       \end{displaymath}
     \end{enumerate}
   \end{remark}

   \begin{definition}\label{def:operations}
     For $f=\sum_{\alpha\in A}a_\alpha\cdot t^\alpha,\; g=\sum_{\beta\in
       B}b_\beta\cdot t^\beta\in\K$ we define the functions 
     \begin{displaymath}
       f+g:\R\longrightarrow\C:\alpha\mapsto f(\alpha)+g(\alpha)
     \end{displaymath}
     and 
     \begin{displaymath}
       f\cdot g=\sum_{\gamma\in A*B} \left(\sum_{\alpha\in A,\beta\in B\;:\;\alpha+\beta=\gamma}
       a_\alpha\cdot b_\beta\right) \cdot t^\gamma.
     \end{displaymath}
   \end{definition}

   \begin{remark}
     With the notation of Definition \ref{def:operations} we obviously
     have that
     \begin{displaymath}
       f+g=\sum_{\gamma\in A\cup B} (a_\gamma+b_\gamma)\cdot t^\gamma,
     \end{displaymath}
     if we use the convention that $a_\gamma=0$ if $\gamma\not\in A$
     and $b_\gamma=0$ whenever $\gamma\not\in B$, and both $f+g$ and
     $f\cdot g$ are elements of $\K$.
     In particular, $(\K,+,\cdot)$ is a subfield of $K$. Moreover, the
     valuation on $K$ induces the valuation
     \begin{displaymath}
       \val:(\K^*,\cdot)\longrightarrow(\R,+):f\mapsto\min\{\alpha\in\R\;|\;f(\alpha)\not=0\}
     \end{displaymath}
     on $\K$, i.e.\ $\val$ is a group homomorphism such that
     \begin{displaymath}
       \val(f+g)\geq\min\{\val(f),\val(g)\}
     \end{displaymath}
     for $f,g\in\K^*$. We call $\lc(f)=f\big(\val(f)\big)$ the leading
     coefficient of $f$, and as usual we extend the valuation to the whole
     of $\K$ by $\val(0)=\infty$. 
   \end{remark}

   \begin{remark}
     If $f_0,\ldots,f_n\in\K$ with $\val(f_i)>0$ for all
     $i=0,\ldots,n$ and $G\in\C[[z_0,\ldots,z_n]]$ is a formal power
     series, then we may actually substitute $z_i$ by $f_i$ in order
     to receive an element $G(f_0,\ldots,f_n)$ in $\K$.
   \end{remark}

   The aim of this paper is to prove the following theorem.

   \begin{theorem}\label{thm:main}
     The field $\K$ is algebraically closed.
   \end{theorem}

   We do neither claim any originality for the definition of the field, nor for the
   fact that it is algebraically closed. In fact, the field can be
   viewed as a special case of much wider classes of fields studied in
   \cite{Ray74} respectively in \cite{Rib92}, and they also show that
   the fields in question are algebraically
   closed. \cite[Thm.~2]{Ray74} and \cite[(5.2)]{Rib92} both reduce
   this fact to general results in the ramification theory of
   non-archimedian valued fields. We want to present a 
   different proof. The basic idea is as follows: Given a non-constant
   polynomial over $\K$ we have to find a root. Using the Weierstra{\ss}'
   Preparation Theorem (see e.g.\ \cite[Kap.\ I,\§ 4]{GR71}) we reduce to the situation
   where the $t$-Newton  polygon (see Notation \ref{not:newton}) has
   only a single lower face connecting the two coordinate 
   axes, and to this polynomial we apply an adaptation of 
   the classical Newton-Puiseux algorithm (see e.g.\
   \cite[Thm.~5.1.14]{DP00}). The
   idea for the reduction step is due to Marina Viazovska.

   Let us fix some notation before we start with the actual
   proof. 

   \begin{notation}\label{not:newton}
     Let $F=\sum_{i=0}^n f_i\cdot y^i\in\K[y]$ be a polynomial of
     degree $n$. We define the \emph{$t$-support} of $F$ as the set
     \begin{displaymath}
       \tsupp(F)=\big\{\big(i,\val(f_i)\big)\;|\;i=0,\ldots,n,\;
       f_i\not=0\big\}\subset\R^2,
     \end{displaymath}
     and we call the convex hull, say $N(F)$, of $\tsupp(F)$ the
     \emph{$t$-Newton polygon} of $F$.
     
     Assume now that $\val(f_i)\geq 0$ for all $i$ and 
     fix a real number $\omega$. We then call
     \begin{displaymath}
       \ord_\omega(F)=\min\{\val(f_i)+\omega\cdot i\;|\;f_i\not=0\}
     \end{displaymath}
     the \emph{$\omega$-order} of $F$, and we define the \emph{$t$-initial form} $F$
     with respect to $\omega$ as
     \begin{displaymath}
       \tin_\omega(F)=\sum_{i\;:\;\val(f_i)+\omega\cdot i=\ord_\omega(F)}
       \lc(f_i)\cdot y^i\in \C[y].
     \end{displaymath}
   \end{notation}

   \begin{example}
     Consider the polynomial
     \begin{displaymath}
       F=\big(2t+t^{\frac{4}{3}}\big)\cdot y^6+y^5+\frac{1}{1-t^{\frac{1}{2}}}\cdot y^4
       -t^{\pi}\cdot y^3+t\cdot y^2+\big(t^{\frac{4e}{5}}-t^4\big)\cdot y+3t^{\frac{5}{2}}\in\K[y].
     \end{displaymath}
     The $t$-support of $F$ is
     \begin{displaymath}
       \tsupp(F)=\left\{\left(0,2.5\right),\left(1,0.8\cdot e\right),
         \left(2,1\right),\left(3,\pi\right),\left(4,0\right),\left(5,0\right),
         \left(6,1\right)\right\},
     \end{displaymath}
     and thus the $t$-Newton polygon $N(F)$ of $F$ looks as follows:
     \medskip
     
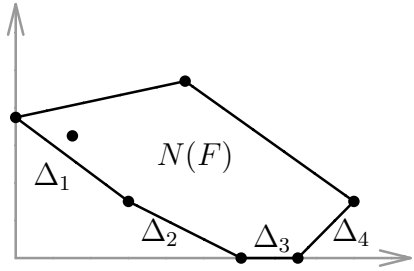
\begin{center}
       \begin{texdraw}
         \drawdim cm  \relunitscale 0.5 \arrowheadtype t:V
         \linewd 0.06  \lpatt (1 0)       
         \setgray 0.6
         \relunitscale 1.5
         \move (0 0) \avec (7 0) \move (0 0) \avec (0 4.5)
         \htext (0.3 1.2){$\Delta_1$}
         \htext (2.2 0.2){$\Delta_2$}
         \htext (4.2 0.1){$\Delta_3$}
         \htext (5.6 0.2){$\Delta_4$}
         \htext (2.5 1.5) {$N(F)$}
         \setgray 0
         \move (0 2.5) \lvec (2 1) \lvec (4 0) \lvec (5 0) \lvec (6 1)
         \lvec (3 3.14) \lvec (0 2.5)
         \move (0 2.5) \fcir f:0 r:0.1
         \move (1 2.17) \fcir f:0 r:0.1
         \move (2 1) \fcir f:0 r:0.1
         \move (3 3.14) \fcir f:0 r:0.1
         \move (4 0) \fcir f:0 r:0.1
         \move (5 0) \fcir f:0 r:0.1
         \move (6 1) \fcir f:0 r:0.1
       \end{texdraw}
     \end{center}
     \medskip
     It has four lower faces $\Delta_1$, $\Delta_2$, $\Delta_3$, and
     $\Delta_4$, and the slope of $\Delta_1$ is $-\frac{3}{4}$. If we
     choose $\omega=\frac{3}{4}$ then $\ord_\omega(F)=\frac{5}{2}$ 
     and $\tin_\omega(F)=y^2+3$.
   \end{example}

   \begin{proof}[Proof of Theorem \ref{thm:main}]
     Consider a non-constant polynomial
     \begin{displaymath}
       F_0=\sum_{i=0}^n f_{i,0}\cdot y^i \in \K[y]
     \end{displaymath}
     with coefficients $f_{i,0}\in\K$. We have to show 
     that there is a $\overline{y}\in\K$ such that
     $F_0(\overline{y})=0$. For this we first want to show that we
     may assume that the coefficients of $F_0$ satisfy certain
     assumptions. 
     
     If $f_{0,0}=0$ then $\overline{y}=0$ will do, so
     that we may assume 
     \begin{equation}\label{eq:cond:1}
       f_{0,0}\not=0.
     \end{equation}
     Multiplying $F_0$ by $t^{-\min\{\val(f_{i,0})\;|\;i=0,\ldots,n\}}$ does not change
     the set of roots of $F_0$ but it allows us to assume that
     \begin{equation}\label{eq:cond:2}
       \val(f_{i,0})\geq 0 \;\;\;\mbox{ for all }\;\;\;i=1,\ldots,n
     \end{equation}
     and that the minimum 
     \begin{displaymath}
       r=\min\{i\;|\;\val(f_{i,0})=0\}
     \end{displaymath}
     exists. 

     We claim that we may actually assume 
     \begin{equation}\label{eq:cond:3}
       r\geq 1.
     \end{equation}
     Suppose the contrary, i.e.\
     $r=0$. If $\val(f_{n,0})>0$ then we can replace $F_0$ by $G=y^n\cdot
     F_0\big(\frac{1}{y}\big)=\sum_{i=0}^n f_{n-i,0}\cdot y^i\in
     \K[y]$ which is a polynomial whose constant coefficient has
     positive valuation, and if we find a root $y'$ of $G$, necessarily
     non-zero, then $\overline{y}=\frac{1}{y'}$ is a root of $F_0$. If
     instead also $\val(f_{n,0})=0$ this replacement would not
     help. However, in this situation 
     \begin{displaymath}
       h={F_0}_{|t=0}=\sum_{i=0}^n f_{i,0}(0)\cdot y^i\in \C[y]
     \end{displaymath}
     is a polynomial of degree $n$ with non-zero constant term. Since
     $\C$ is algebraically closed there is a $0\not=c\in \C$ such
     that $h(c)=0$. If we then set 
     \begin{displaymath}
       G=F_0(y+c)=\sum_{i=0}^n\left(\sum_{j=i}^n
         f_{j,0}\cdot\binom{j}{i}\cdot c^{j-i}\right)\cdot y^i,
     \end{displaymath}
     the constant coefficient, say $g_0\in \K$, of this polynomial
     satisfies $g_0(0)=h(c)=0$ and has thus positive valuation. We may
     again replace $F_0$ by $G$, and if $y'$ is a root of $G$
     then $\overline{y}=y'+c$ is a root of $F_0$. This shows the
     claim. 

     We are now ready to show by induction on $n+r$ that a
     polynomial $F_0$ satisfying the conditions \eqref{eq:cond:1},
     \eqref{eq:cond:2} and \eqref{eq:cond:3} has a root
     $\overline{y}\in\K$ such that $\val(\overline{y})> 0$.

     If $n=r=1$ there is nothing to show, and we may assume that $n>1$.

     Due to the above assumptions the $t$-Newton polygon of $F_0$
     looks basically as follows:
     
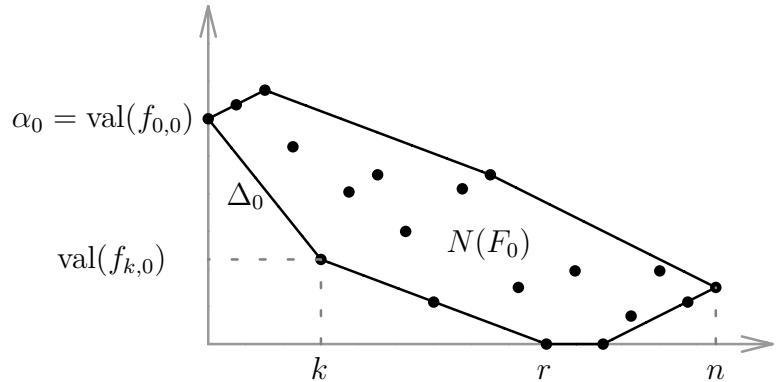
\begin{center}
       \begin{texdraw}
         \drawdim cm  \relunitscale 0.5 \arrowheadtype t:V
         \linewd 0.06  \lpatt (1 0)       
         \setgray 0.6
         \relunitscale 1.5
         \move (0 0) \avec (10 0) \move (0 0) \avec (0 6)
         \htext (0.3 2.4){$\Delta_0$}
         \htext (4.2 1.5) {$N(F_0)$}
         \setgray 0
         \move (0 4) \lvec (2 1.5) \lvec (6 0) \lvec (7 0) \lvec (9 1)
         \lvec (5 3) \lvec (1 4.5) \lvec (0 4)
         \move (0 4) \fcir f:0 r:0.1
         \move (0.5 4.25) \fcir f:0 r:0.1
         \move (1 4.5) \fcir f:0 r:0.1
         \move (1.5 3.5) \fcir f:0 r:0.1
         \move (2 1.5) \fcir f:0 r:0.1
         \move (2.5 2.7) \fcir f:0 r:0.1
         \move (3 3) \fcir f:0 r:0.1
         \move (3.5 2) \fcir f:0 r:0.1
         \move (4 0.75) \fcir f:0 r:0.1
         \move (4.5 2.75) \fcir f:0 r:0.1
         \move (5 3) \fcir f:0 r:0.1
         \move (5.5 1) \fcir f:0 r:0.1
         \move (6 0) \fcir f:0 r:0.1
         \move (6.5 1.3) \fcir f:0 r:0.1
         \move (7 0) \fcir f:0 r:0.1
         \move (7.5 0.5) \fcir f:0 r:0.1
         \move (8 1.3) \fcir f:0 r:0.1
         \move (8.5 0.75) \fcir f:0 r:0.1
         \move (9 1) \fcir f:0 r:0.1        
         \htext (-3.5 3.7){$\alpha_0=\val(f_{0,0})$}
         \htext (-2.7 1.2){$\val(f_{k,0})$}
         \htext (1.8 -0.6){$k$}
         \htext (5.8 -0.6){$r$}
         \htext (8.8 -0.6){$n$}
         \setgray 0.5
         \lpatt (0.1 0.4) 
         \move (2 0) \lvec (2 1.5) \lvec (0 1.5)
         \move (9 0) \lvec (9 1)         
       \end{texdraw}
     \end{center}
     Here we simply set $\alpha_0=\val(f_{0,0})$ and choose $k$ such
     that the point $\big(k,\val(f_{k,0})\big)$ is the second end
     point of the lower face, say $\Delta_0$, of the $t$-Newton polygon emanating from
     the vertex $(0,\alpha_0)$. By our assumptions we have necessarily
     \begin{equation}\label{eq:proof:1}
       k\leq r.
     \end{equation}
     If we now set
     \begin{displaymath}
       \omega_0=\frac{\alpha_0-\val(f_{k,0})}{k}>0,
     \end{displaymath}
     then $-\omega_0$ is the slope of the above mentioned
     $\Delta_0$. With this notation we can write the $t$-initial form
     $F_0$ with respect to $\omega_0$ as follows:
     \begin{displaymath}
       \tin_{\omega_0}(F_0)=
       \sum_{i=0}^n \big(t^{i\cdot \omega_0-\alpha_0} \cdot
       f_{i,0}\big)_{|t=0} \cdot y^i \in \C[y].
     \end{displaymath}
     In particular, the degree of the $t$-initial form with respect to
     $\omega_0$ is
     \begin{equation}\label{eq:proof:2}
       \deg\big(\tin_{\omega_0}(F_0)\big)=k,
     \end{equation}
     and the constant coefficient is $\lc(f_{0,0})\not=0$.
     Since $\C$ is algebraically closed we can choose a non-zero
     root of $\tin_{\omega_0}(F_0)$, 
     or more precisely 
     \begin{displaymath}
       \exists\;0\not=c_0\in\C\;\;\mbox{ and }\;\;0<r'\leq r
       \;\;\mbox{ s.t. }\;\;
       \tin_{\omega_0}(F_0)=(y-c_0)^{r'}\cdot g, \;\; g(c_0)\not=0,
     \end{displaymath}
     i.e.\ $c_0$ is a root of multiplicity $r'$ of
     $\tin_{\omega_0}(F_0)$ and $r'\leq r$ follows from
     \eqref{eq:proof:1} and \eqref{eq:proof:2}.

     Having found this root $c_0$ we transform $F_0$ into a new
     polynomial
     \begin{displaymath}
       F_1=t^{-\alpha_0}\cdot F_0\big(t^{\omega_0}\cdot (y+c_0)\big)=
       \sum_{i=0}^n f_{i,1}\cdot y^i\in\K[y].
     \end{displaymath}
     The coefficients $f_{i,1}$ of $F_1$ are just
     \begin{equation}\label{eq:coefF1}
       f_{i,1}=\sum_{j=i}^n f_{j,0}\cdot t^{j\cdot \omega_0-\alpha_0}
       \cdot \binom{j}{i}\cdot c_0^{j-i}
     \end{equation}
     for $i=0,\ldots,n$. In particular they have all non-negative
     valuation. But for the first $r'$ coefficients we
     know more, namely 
     \begin{displaymath}
       f_{i,1}(0)=\frac{1}{i!}\cdot \frac{\partial^i
         \tin_{\omega_0}(F_0)}{\partial y^i}(c_0)
       \left\{
         \begin{array}[m]{ll}
           =0, & \mbox{ if } 0\leq i\leq r'-1,\\
           \not=0, & \mbox{ if } i=r'.
         \end{array}
       \right.
     \end{displaymath}
     Note that we here use that the characteristic of the ground field
     is zero! It follows that the number $r'$ defined above plays the
     same role for $F_1$ as $r$ does for $F_0$, i.e.
     \begin{displaymath}
       r'=\min\{i\;|\;\val(f_{i,1})=0\},
     \end{displaymath}
     and as we have seen before $r'$ satisfies the inequalities
     \begin{displaymath}
       1\leq r'\leq k\leq r.
     \end{displaymath}
     If we find a root $y'\in\K$ of $F_1$ then
     $\overline{y}=t^{\omega_0}\cdot (y'+c_0)\in\K$ will be a root of
     $F_0$ with $\val(\overline{y})=\omega_0+\min\{0,\val(y')\}$. 

     In particular, if $f_{0,1}=0$ then $y'=0$ will do and we
     are done since then $\val(\overline{y})\geq \omega_0>0$. We may
     therefore assume that $f_{0,1}\not=0$, so that 
     $F_1$ satisfies the assumption \eqref{eq:cond:1},
     \eqref{eq:cond:2} and \eqref{eq:cond:3}. Thus, if $r'<r$ we
     are done by induction since $\deg(F_1)=n$, and we may assume
     therefore that
     \begin{equation}\label{eq:r=r'}
       r'=r.
     \end{equation}
     Note that this forces $k=r$, i.e.\ the $t$-Newton polygon of
     $F_0$ actually looks as follows,
     
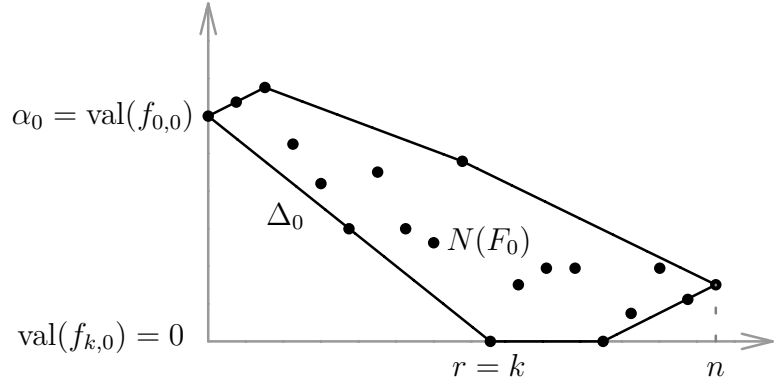
\begin{center}
       \begin{texdraw}
         \drawdim cm  \relunitscale 0.5 \arrowheadtype t:V
         \linewd 0.06  \lpatt (1 0)       
         \setgray 0.6
         \relunitscale 1.5
         \move (0 0) \avec (10 0) \move (0 0) \avec (0 6)
         \htext (1 2){$\Delta_0$}
         \htext (4.2 1.5) {$N(F_0)$}
         \setgray 0
         \move (0 4) \lvec (5 0)  \lvec (7 0) \lvec (9 1)
         \lvec (4.5 3.2) \lvec (1 4.5) \lvec (0 4)
         \move (0 4) \fcir f:0 r:0.1
         \move (0.5 4.25) \fcir f:0 r:0.1
         \move (1 4.5) \fcir f:0 r:0.1
         \move (1.5 3.5) \fcir f:0 r:0.1
         \move (2 2.8) \fcir f:0 r:0.1
         \move (2.5 2) \fcir f:0 r:0.1
         \move (3 3) \fcir f:0 r:0.1
         \move (3.5 2) \fcir f:0 r:0.1
         \move (4 1.75) \fcir f:0 r:0.1
         \move (4.5 3.2) \fcir f:0 r:0.1
         \move (5 0) \fcir f:0 r:0.1
         \move (5.5 1) \fcir f:0 r:0.1
         \move (6 1.3) \fcir f:0 r:0.1
         \move (6.5 1.3) \fcir f:0 r:0.1
         \move (7 0) \fcir f:0 r:0.1
         \move (7.5 0.5) \fcir f:0 r:0.1
         \move (8 1.3) \fcir f:0 r:0.1
         \move (8.5 0.75) \fcir f:0 r:0.1
         \move (9 1) \fcir f:0 r:0.1        
         \htext (-3.5 3.7){$\alpha_0=\val(f_{0,0})$}
         \htext (-3.4 -0.2){$\val(f_{k,0})=0$}
         \htext (4.3 -0.6){$r=k$}
         \htext (8.8 -0.6){$n$}
         \setgray 0.5
         \lpatt (0.1 0.4) 
         \move (9 0) \lvec (9 1)         
       \end{texdraw}
     \end{center}
     and the lower face $\Delta_0$ of the $t$-Newton polygon of $F_0$ emanating
     from $(0,\alpha_0)$ connects the two coordinate axes. 

     We now claim that in this situation we may indeed assume that 
     \begin{equation}
       \label{eq:cond:4}
       r=n.
     \end{equation}
     For this define the polynomial
     \begin{displaymath}
       F'=\frac{F_1}{f_{r,0}(0)}=\frac{t^{-\alpha_0}}{f_{r,0}(0)}\cdot
       F_0\Big(t^{\omega_0}\cdot (y+c_0)\Big)
       =y^r+\sum_{i=0}^n f'_i \cdot y^i\in \K[y],
     \end{displaymath}
     and note that $f'_i\in\K$ with $\val(f'_i)>0$ for
     all $i=0,\ldots,n$ since $c_0$ is a root of
     $\tin_{\omega_0}(F_0)$ of order $r$. Moreover, we consider the polynomial     
     \begin{displaymath}
       F''=y^r+\sum_{i=0}^n z_i\cdot y^i\in\C[[z_0,\ldots,z_n,y]]
     \end{displaymath}
     as a formal power series over $\C$, which then is \emph{regular of order
       $r$ in $y$} in the sense of the Weierstra{\ss}' Preparation
     Theorem (see e.g.\ \cite[Kap.\ I,\§ 4]{GR71} or 
     \cite[Thm.\ 3.2.4]{DP00}). The latter theorem thus implies that
     there exists a unit $U\in\C[[z_0,\ldots,z_n,y]]^*$ and a
     Weierstra{\ss} polynomial $P=y^r+\sum_{i=0}^{r-1} p_i\cdot y^i\in
     \C[[z_0,\ldots,z_n]][y]$ with $p_i(0)=0$ for all $i=0,\ldots,r-1$ such that
     \begin{displaymath}
       F''=U\cdot P.
     \end{displaymath}
     Since the $f'_i$ have strictly positive valuation we can
     substitute the $z_i$ by $f'_i$ in $U$ and $P$ to get 
     an invertible power series
     \begin{displaymath}
       U'=U(f'_0,\ldots,f'_n,y)\in\K[[y]]^*
     \end{displaymath}
     and a polynomial
     \begin{displaymath}
       P'=P(f'_0,\ldots,f'_n,y)=y^r+\sum_{i=0}^{r-1}
       p_i(f'_0,\ldots,f'_n)\cdot y^i\in \K[y]
     \end{displaymath}
     with
     \begin{equation}\label{eq:valcoef}
       \val\big(p_i(f'_0,\ldots,f'_n)\big)>0\;\;\;\mbox{ for all }i=0,\ldots,r-1.
     \end{equation}
     If $p_0(f'_0,\ldots,f'_n)=0$ then $F'(0)=U'(0)\cdot P'(0)=0$ and
     thus $\overline{y}=c_0\cdot t^{\omega_0}$ is a root of
     $F_0$ with $\val(\overline{y})=\omega_0>0$, so that we are done. Otherwise
     $P'$ satisfies the conditions \eqref{eq:cond:1}, 
     \eqref{eq:cond:2} and \eqref{eq:cond:3}. Thus, if $r<n$ then
     $r+r<r+n$ and by induction there exists a $y'\in\K$
     such that $P'(y')=0$ with  $\val(y')> 0$. 
     Since its valuation is positive we can substitute
     $y'$ into $U'$ and get an element
     $U'(y')\in\K$. But then
     $F'(y')=U'(y')\cdot P'(y')=0$, and
     hence 
     \begin{displaymath}
       \overline{y}=t^{\omega_0}\cdot (y'+c_0)\in\K
     \end{displaymath}
     is a root of $F_0$ with $\val(\overline{y})=\omega_0>0$. This
     proves the claim.

     We finally claim that  under the assumption \eqref{eq:cond:4} we
     can also assume
     \begin{equation}
       \label{eq:cond:5}
       f_{n,0}=1.
     \end{equation}
     For this note that $r=n$ implies that $\val(f_{n,0})=0$, i.e.\
     $f_{n,0}$ is a unit in the valuation ring of $\K$ and
     $\frac{1}{f_{n,0}}$ has valuation zero as well. Thus, replacing
     $F_0$ by $\frac{F_0}{f_{n,0}}$ does not effect the conditions
     \eqref{eq:cond:1}, \eqref{eq:cond:2}, \eqref{eq:cond:3}, or
     \eqref{eq:cond:4}. This shows the claim. 

     Note that if $F_0$ satisfies \eqref{eq:cond:4} and
     \eqref{eq:cond:5} then by \eqref{eq:coefF1} and \eqref{eq:r=r'}
     $F_1$ satisfies the corresponding conditions as well. Thus,
     applying the same procedure to $F_1$ and going on by recursion we
     may assume that we produce for each $\nu\in\N$ a polynomial
     \begin{displaymath}
       F_\nu=\sum_{i=0}^n f_{i,\nu}\cdot y^i\in\K[y]
     \end{displaymath}
     satisfying the corresponding versions of \eqref{eq:cond:1}, \eqref{eq:cond:2},
     \eqref{eq:cond:3}, \eqref{eq:cond:4}, and \eqref{eq:cond:5}, and
     we produce a root $0\not=c_\nu\in\C$ of $\tin_{\omega_\nu}(F_\nu)$ of order 
     \begin{displaymath}
       n=\min\{i\;|\;\val(f_{i,\nu})=0\}
     \end{displaymath}
     such that for $i=0,\ldots,n$
     \begin{equation}\label{eq:coefFnu}
       f_{i,\nu}=\sum_{j=i}^n f_{j,\nu-1}\cdot t^{j\cdot
         \omega_{n-1}-\alpha_{\nu-1}}\cdot \binom{j}{i}\cdot
       c_{\nu-1}^{j-i},
     \end{equation}
     where for each $\nu\in\N$
     \begin{displaymath}
       \alpha_\nu=\val(f_{0,\nu})>0
     \end{displaymath}
     and 
     \begin{equation}\label{eq:omega}
       \omega_{\nu}=\frac{\alpha_\nu}{n}>0
     \end{equation}
     is the negative of the slope of the lower face, say $\Delta_\nu$,
     of the $t$-Newton polygon of $F_\nu$ connecting the two coordinate axes by
     joining the points $(0,\alpha_\nu)$ and $(n,0)$.     
     Note that for this we use the fact that if at some
     point $F_\nu(0)=0$ then
     \begin{displaymath}
       \overline{y}=\sum_{i=0}^{\nu-1} c_i\cdot t^{\omega_0+\ldots+\omega_i}
       =t^{\omega_0}\cdot\Big(c_0+t^{\omega_1}\cdot\big(c_1+\cdots
       t^{\omega_{\nu-2}}\cdot(c_{\nu-2}+t^{\omega_{\nu-1}}\cdot c_{\nu-1})\cdots\big)\Big)
     \end{displaymath}
     is a root of $F_0$ of valuation $\omega_0>0$.

     That way we obviously construct a generalised Laurent series
     \begin{displaymath}
       \overline{y}=\sum_{i=0}^\infty c_i\cdot
       t^{\omega_0+\ldots+\omega_i}
     \end{displaymath}
     in the field $K$, and it remains to show that indeed
     $\overline{y}\in\K$ and $F_0(\overline{y})=0$.

     Let us first address the issue that $\overline{y}\in\K$.
     By \eqref{eq:omega} we know that $n\cdot\omega_\nu-\alpha_\nu=0$
     and $(n-1)\cdot \omega_\nu-\alpha_\nu=-\omega_\nu$, so that
     \eqref{eq:coefFnu} and the fact that $f_{n,\nu}=1$ imply that
     \begin{displaymath}
       f_{n-1,\nu+1}=f_{n-1,\nu}\cdot t^{-\omega_\nu}+n\cdot c_\nu,
     \end{displaymath}
     or equivalently
     \begin{displaymath}
       f_{n-1,\nu}=-n\cdot c_\nu\cdot t^{\omega_\nu}+t^{\omega_\nu}\cdot f_{n-1,\nu+1}.
     \end{displaymath}
     Doing a descending induction on $\nu$ we deduce
     \begin{displaymath}
       f_{n-1,0}=-n\cdot \sum_{i=0}^\nu c_i\cdot  t^{\omega_0+\ldots+\omega_i}+
       t^{\omega_0+\ldots+\omega_\nu}\cdot f_{n-1,\nu+1}.
     \end{displaymath}
     Since the valuation of $f_{n-1,\nu+1}$ is strictly positive it
     follows that the first $\nu+1$ summands of $f_{n-1,0}$ coincide
     with $-n\cdot \sum_{i=0}^\nu c_i\cdot
     t^{\omega_0+\ldots+\omega_i}$, and since this holds for each
     $\nu\in\N$ we necessarily have
     \begin{displaymath}
       \overline{y}=-\frac{f_{n-1,0}}{n}\in\K.
     \end{displaymath}
     Note that we here again use that the characteristic of $\C$ is zero.
    
     In order to show that $F_0(\overline{y})=0$ we set
     \begin{displaymath}
       \overline{y}_\nu=\sum_{i=\nu}^\infty c_i\cdot t^{\omega_\nu+\ldots+\omega_i},
     \end{displaymath}
     so that
     \begin{displaymath}
       F_\nu(\overline{y}_\nu)=t^{\alpha_\nu}\cdot F_{\nu+1}(\overline{y}_{\nu+1}).
     \end{displaymath}
     But this equation together with a simple induction shows that
     \begin{displaymath}
       \val\big(F_0(\overline{y})\big)=\sum_{i=0}^\nu \alpha_i +
       \val\big(F_{\nu+1}(\overline{y}_{\nu+1})\big)
     \end{displaymath}
     for each $\nu\in\N$. Since the coefficients of $F_{\nu+1}$ all
     have non-negative valuation and since
     $\val(\overline{y}_{\nu+1})>0$ it follows that the last summand
     is non-negative, and therefore
     \begin{displaymath}
       \val\big(F_0(\overline{y})\big)\geq n\cdot \sum_{i=0}^\nu
       \omega_i \stackrel{\nu\rightarrow\infty}{\longrightarrow}\infty.
     \end{displaymath}
     This implies that $F_0(\overline{y})=0$ and finishes the proof.
   \end{proof}

   \begin{remark}
     \begin{enumerate}
     \item If we replace the base field $\C$ in the definition of $\K$
       by any algebraically closed field of characteristic zero, then
       the Theorem \ref{thm:main} holds with the same proof.
     \item If we replace the base field $\C$ by a field of positive
       characteristic $p$ Theorem \ref{thm:main} holds no longer. The
       Artin-Schreyer polynomial 
       \begin{displaymath}
         F=y^p-y-\frac{1}{t}
       \end{displaymath}
       has the roots
       \begin{displaymath}
         \overline{y}=-k+\sum_{i=1}^\infty t^{-\frac{1}{p^i}}\in K\setminus\K,
       \end{displaymath}
       for $k=0,\ldots,p-1$ (see \cite{Abh56}). The algebraic closure
       of the quotient field of the formal power series ring, i.e. the
       analogue of $\C\{\{t\}\}$ in this situation is studied in
       \cite{Ked01}. Since already the square of any of the roots of
       the above polynomial $F$ has a support which can no longer be
       written as a single ascending sequence, there is no nice
       substitution of the analogue of $K$ for tropical geometry in
       positive characteristic.
     \item The valuation ring 
       \begin{displaymath}
         R_{\val}=\{f\in\K\;|\;\val(f)\geq 0\}
       \end{displaymath}
       of $\K$ is non-noetherian local ring of dimension one with maximal
       ideal $\m=\langle t^\alpha\;|\;\alpha\in\R_{>0}\rangle$.
     \item The field extension $\C\subset\K$ has infinite
       transcendence degree, since whenever $\alpha_1,\ldots,\alpha_n$
       are algebraically independent over $\Q$ then
       $t^{\alpha_1},\ldots,t^{\alpha_n}$ are algebraically
       independent over $\C$.
     \end{enumerate}
   \end{remark}

   \begin{definition}
     The valuation on $\K$ induces via the exponential map the norm
     \begin{displaymath}
       |\cdot|:\K\longrightarrow\R:f\mapsto\exp\big(-\val(f)\big)
     \end{displaymath}
     on $\K$, where we use the convention that $\exp(-\infty)=0$. It
     satisfies the strong triangle inequality
     \begin{displaymath}
       |f+g|\leq\max\{|f|,|g|\}.
     \end{displaymath}
     As usual we call a sequence $(f_n)_{n\in\N}$ in $\K$ a
     \emph{Cauchy sequence} with respect to $|\cdot|$ if for all
     $\varepsilon>0$ there exists an $N(\varepsilon)\in\N$ such that
     $|f_n-f_m|<\varepsilon$ for all $n,m\geq N(\varepsilon)$. And we
     call a sequence $(f_n)_{n\in\N}$ in $\K$ \emph{convergent} with respect
     to $|\cdot|$ if there is an $f\in\K$ such that for all
     $\varepsilon>0$ there exists an $N(\varepsilon)\in\N$ such that  
     $|f_n-f|<\varepsilon$ for all $n\geq N(\varepsilon)$. 
   \end{definition}

   Compared with the field $\C\{\{t\}\}$ of Puiseux series the field
   $\K$ has the advantage that it is complete with respect to the
   norm induced by the valuation.

   \begin{proposition}\label{prop:complete}
     $(\K,|\cdot|)$ is complete, i.e.\ every Cauchy sequence is convergent.
   \end{proposition}
   \begin{proof}
     Let $(f_n)_{n\in\N}$ be a Cauchy sequence. Given any positive
     integer $M$ we set $\varepsilon_M=\exp(-M)>0$. Thus there is an
     $N(\varepsilon_M)\in\N$ such that 
     \begin{displaymath}
       \exp(-M)=\varepsilon_M>|f_n-f_m|=\exp\big(-\val(f_n-f_m)\big)
     \end{displaymath}
     for all $n,m\geq N(\varepsilon_M)$, or equivalently
     \begin{displaymath}
       \val(f_n-f_m)>M.
     \end{displaymath}
     This implies that $f_n(\alpha)=f_m(\alpha)$ for all $\alpha\leq
     M$ and for all $n,m\geq N(\varepsilon_M)$. Without loss of
     generality we may assume that $N(\varepsilon_M)\geq
     N(\varepsilon_{M'})$ whenever $M\geq M'$. We may therefore
     define a function $f:\R\longrightarrow \C$ by
     \begin{displaymath}
       f(\alpha)=f_{N(\varepsilon_M)}(\alpha)
     \end{displaymath}
     if $\alpha\geq M$, and obviously $(f_n)_{n\in\N}$ converges to this
     function $f$. 
   \end{proof}

   \begin{remark}
     \begin{enumerate}
     \item The statement of Proposition \ref{prop:complete} remains
       true if we replace in the definition of $\K$ the field $\C$ by
       any other field. It is independent of the characteristic.
     \item If we replace in the definition of $\K$ the domain $\R$ 
       of the elements in $\K$ by $\Q$ we get the completion of
       $\C\{\{t\}\}$ with respect to the norm induced by the
       valuation. With the same proof as in Theorem \ref{thm:main}
       this field is algebraically closed. But note that the value
       group is still only $Q$.
     \end{enumerate}
   \end{remark}

   
\providecommand{\bysame}{\leavevmode\hbox to3em{\hrulefill}\thinspace}

\end{document}